\documentclass[12 pt]{amsart}
\author{Mikhail Ganzhinov}
\address{Department of Communications and Networking, Aalto University School of Electrical Engineering, P.O.\ Box 15400, 00076 Aalto, Finland}

\usepackage{graphicx}
\usepackage{float}
\usepackage[margin=1.0in]{geometry}
 \usepackage{amsfonts}
\usepackage{amsmath}
\usepackage{amsthm}
\usepackage{xcolor}
\usepackage{hyperref}
\usepackage{booktabs}

\newtheorem{proposition}{Proposition}
\newtheorem{definition}{Definition}
\newtheorem{lemma}{Lemma}

\theoremstyle{remark}
\newtheorem*{remark}{Remark}
\newcommand{\U}{\textrm{U}}
\newcommand{\OO}{\textrm{O}}
\newcommand{\GL}{\textrm{GL}}
\newcommand{\PGL}{\textrm{PGL}}

\newcommand{\Sym}{\textrm{Sym}}

\newcommand{\Tr}{\textrm{Tr}}
\newcommand{\Dim}{\textrm{dim}}
\title{Highly symmetric lines}
\begin{document}
\maketitle
\begin{abstract}
A generalization of highly symmetric frames is presented by considering also projective stabilizers of frame vectors. This allows construction of highly symmetric line systems and study of highly symmetric frames in a more unified manner. Construction of highly symmetric line systems involves computation of twisted spherical functions associated with finite groups. Further generalizations include definition of highly symmetric systems of subspaces. We give several examples which illustrate our approach including 3 new kissing configurations which improve lower bounds on the kissing number in $d=10,11,14$ to 510, 592 and 1932 respectively.
\end{abstract}

\section{Introduction}

Finding configurations of subspaces with desirable properties is important in many applications such as communication, coding, and quantum information theory \cite{PWTH,SS,SH}. Desirable properties include optimality in various metrics. Configurations with high degree of symmetry are often among the best candidates for optimality \cite{CKMP,CHS,JKM}.

In this paper, by generalizing definition of highly symmetric frames presented in \cite{BW}, we define highly symmetric systems of subspaces. These systems of subspaces are subspace-transitive and determined by the stabilizer subgroups of their subspaces. By focusing mainly on the case of line systems it will be shown that, like in the case of highly symmetric frames, highly symmetric line systems are always defined by irreducible representations of finite groups. Highly symmetric line systems tend to have a smaller number of distinct angles between lines in comparison to line systems corresponding to generic orbits of the same group. This increases the likelyhood of finding a line system with good properties among highly symmetric line systems. There are finitely many highly symmetric line systems associated with any irreducible representation of a finite group.

The paper is organized as follows. In Section 2 we give definitions and basic properties of spherical codes, frames, and systems of subspaces. In Section 3 after defining symmetry groups of frames and systems of subspaces we define highly symmetric systems of subspaces. We prove that highly symmetric line systems (i.e., highly symmetric systems of 1-dimensional subspaces) are always defined by representations of finite groups. We also show that highly symmetric line systems are always related to a certain generalization of highly symmetric frames. In Section 4 we outline construction of generalized highly symmetric frames and other related objects. We illustrate the theory with examples in Section 5. 

\section{Preliminaries}

We begin by recalling definitions and basic properties of spherical codes, frames and systems of subspaces including line systems. See \cite{C,CHS,W} for more details regarding frame theory and subspace packings. Throughout the text we assume that $S$ is a finite set and $\mathcal{H}$ is a $d$-dimensional Hilbert space ($d<\infty$) which is often complex but can also be real.

A \emph{spherical code} is a finite sequence $\mathcal{C}=(v_j)_{j\in S}$ of unit vectors in the Hilbert space $\mathcal{H}$. The set of inner products 
\begin{displaymath}
A(\mathcal{C})=\{\langle v_i,v_j\rangle:i,j\in S, i\neq j\} 
\end{displaymath}
is the \emph{angle set} of this code. When $\mathcal{H}$ is real, the angle set contains cosines of the angles between distinct vectors of the code $\mathcal{C}$, hence the name.

A finite sequence of $k$-dimensional subspaces $\mathcal{S}=(V_j)_{j\in S}$ of the Hilbert space $\mathcal{H}$ is called a \emph{system of subspaces}. When $k=1$ we denote subspaces $V_j$ (which are now lines) as $l_j$ and the corresponding  \emph{line system} as $\mathcal{L}$. Each line $l_j$ in a line system is spanned by a single unit vector $\phi_j$. The angle set of a line system $\mathcal{L}$ can be defined similarly to the case of spherical codes:
\begin{displaymath}
A(\mathcal{L})=\{|\langle \phi_i,\phi_j\rangle|: i,j\in S,i\neq j, \phi_i\in l_i, \phi_j\in l_j, \rVert\phi_i\lVert_2=\rVert\phi_j\lVert_2=1\}.
\end{displaymath}
The previous definition does not depend on the choice of unit vectors representing lines.

A finite sequence of vectors $\Phi=(v_j)_{j\in S}$ in the Hilbert space $\mathcal{H}$ is called a \emph{frame} if it spans $\mathcal{H}$. This is equivalent to existence of such constants $A,B>0$ that inequailities
\begin{displaymath}
A\rVert v\lVert_2^2\leq\sum_{j\in S}|\langle v,v_j\rangle|^2\leq \rVert v\lVert_2^2
\end{displaymath}
hold for any $v\in\mathcal{H}$. If one can choose $A=B$ we call $\Phi$ a \emph{tight} frame. This is equivalent to the following generalization of orthogonal expansion
\begin{displaymath}
v=\frac{1}{A}\sum_{j\in S}\langle v,v_j\rangle v_j \textrm{ for all } v\in\mathcal{H}
\end{displaymath}
which is often a desired property for a frame. If all vectors of $\Phi$ have the same norm, $\Phi$ is an \emph{equal norm frame} and if all vectors are of unit norm, $\Phi$ is a \emph{unit norm frame}. Unit norm frame can be interpreted as a spherical code. The set of frame angles between vectors of a unit norm frame $\Phi$ is defined as
\begin{displaymath}
A(\Phi)=\{|\langle v,w\rangle|:v,w\in\Phi,v\neq w\}.
\end{displaymath}
In fact, $A(\Phi)=A(\mathcal{L})$ whenever unit vectors $(\phi_j)_{j\in S}$ span lines $(l_j)_{j\in S}$ of a line system $\mathcal{L}$. Vectors of a frame $\Phi$ can be used to construct a \emph{Gram matrix} 
\begin{equation}\label{eq:0}
G_\Phi=[\langle\phi_j,\phi_i\rangle]_{i,j\in S}
\end{equation}
associated with it. This matrix is positive semidefinite and it is possible to recover $\Phi$ from its Gram matrix, up to a unitary equivalence. Any positive semidefinite matrix is a Gram matrix of some frame. The Gram matrix of a tight frame is a projection modulo scalar factor. 
\newline

In all three cases ($\mathcal{C},\mathcal{L},\Phi$) an important problem is to find such packings in $\mathcal{H}$ that vectors or lines are spread far from each other. In other words, the quantity
\begin{displaymath}
\mu=\underset{\alpha\in\textrm{A}(\mathcal{X})}{\textrm{max}}\alpha,
\end{displaymath}
where $\mathcal{X}$ denotes $\mathcal{C}$, $\mathcal{L}$ or $\Phi$, is required to be close to the smallest possible value among all packings of the same size and type. Here, $\mu$ is the smallest angular distance between vectors or lines in $\mathcal{X}$. If $\mathcal{X}=\mathcal{L}\textrm{ or }\Phi$ then $\mu$ is called a \emph{coherence} of a line system or a unit norm frame.

\section{Symmetries of frames and systems of lines}
In this section we consider a situation where a frame or a line system remains invariant under linear or projective linear transformations. We will pay special attention to the situation where a stabilizer subgroup of a line distinguishes it among all other lines. For more information on the representation theory and symmetries of frames see \cite{CR,W} respectively. Note, that we define symmetry group of a frame as a group of linear transformations preserving it, slightly differently as in \cite{W}.

\subsection{Symmetries of frames} 

Two frames $\Phi,\Phi'$ in Hilbert spaces $\mathcal{H},\mathcal{H}'$ respectively are said to be \emph{linearly equivalent} if there exist a linear isomorphism $M$ between $\mathcal{H}$ and $\mathcal{H}'$ such that the sequence of vectors $(M\phi_j)_{j\in S}$ can be reordered to $\Phi'$. Each frame $\Phi$ is linearly equivalent to a tight frame with $A=B=|\Phi|/d$. Up to a unitary (orthogonal if $\mathcal{H}$ is real) factor this equivalence is unique and $M$ can be taken to be
\begin{equation}\label{eq:1}
M=\frac{|\Phi|}{d}\cdot\tilde{M}^{-1/2} \textrm{\hspace{0.7cm} where  \hspace{0.7cm}} \tilde{M}v=\sum_{j\in S}\langle v,\phi_j \rangle \phi_j \textrm{ \hspace{0.0cm}for all }v\in\mathcal{H}.
\end{equation}
The set of all maps $M\in\GL(\mathcal{H})$ which permute vectors of $\Phi$ is called the \emph{symmetry group} of $\Phi$ and is denoted by $\Sym(\Phi)$.
It is simple to verify that $\Sym(\Phi)$ is indeed a group. Each $M\in\Sym(\Phi)$ is uniquely determined by the action on a basis taken from $\Phi$, as a consequence, group $\Sym(\Phi)$ is finite.
If $\Phi$ is a tight frame, then $\Sym(\Phi)$ contains only unitaries.

On practice, instead of the full symmetry group of a frame $\Phi$, a subgroup $G\leq\Sym(\Phi)$ is used. If $v=hv$ for some $v\in\Phi$ and $h\in G$, we say that $h$ \emph{stabilizes} $v$. The set of all elements of $G$ stabilizing $v$ form a \emph{stabilizer subgroup} $H_v$ of a vector $v$. Similarly, if  $c_hv=hv$ where $c_h\in\mathbb{C}$, we say that $h$ stabilizes a line $l$ spanned by $v$. The set of all  elements of $G$ stabilizing  $l$ form a stabilizer subgroup $H_l$ of a line $l$. A frame $\Phi$ is \emph{$G$-transitive} if for any vectors $v_i,v_j\in\Phi$ exists $g\in G$ such that $v_j=gv_i$.  If $\Phi$ is $G$-transitive, the stabilizer subgroup of any other frame vector $gv$ is $H_{gv}=gH_v g^{-1}$.

Let $G\leq\GL(\mathcal{H})$ be a finite group, a frame of the form $(gv)_{g\in G}$ where $v\in\mathcal{H}$ is called a \emph{group frame} or more specifically a \emph{$G$-frame}. In other words, $G$-frame is a frame that is $G$-orbit of a single vector. For any $G$-frame $\Phi$ we have $|\Phi|=|G|$, $G\leq\Sym(\Phi)$ and $\Phi$ is $G$-transitive. Group frame $\Phi$ contains $|H_v|$ copies of each vector $u\in\Phi$ and there are in total $|G|/|H_v|$ distinct vectors inside $\Phi$. Each distinct vector corresponds to a different  $G/H_v$-coset, i.e., is equal to $ghv$ where element $g\in G$ is a representative of a given coset and $h\in H_v$. In order to get rid of unnecessary copies of vectors inside $\Phi$ we can form a new frame $\Phi'=(gv)_{gH_v\in G/H_v}$ which is called \emph{homogeneous}. If $\Phi$ is a tight frame, so is $\Phi'$. By choosing $M$ as in \eqref{eq:1} we can always transform $\Phi$ into linearly equivalent unit norm group tight frame and $\Phi'$ into linearly equivalent homogeneous unit norm tight frame.

 Before generalizing highly symmetric frames, we need to define group representations, representation theory of finite groups will be covered in more details in the next section. A \emph{linear representation} of an abstract group $G$ on a Hilbert space $\mathcal{H}$ is a group homomorphism $\rho:G\rightarrow\GL(\mathcal{H})$. If additionally $\rho(G)\leq\U(\mathcal{H})$ we say that $\rho$ is a \emph{unitary representation}. By slightly abusing notation, the representation $\rho:G\rightarrow\GL(\mathcal{H})$ can be abbreviated by $\mathcal{H}$ or $G$ if the homomorphism $\rho$ is clear from the context.

Let $\rho_1$ and $\rho_2$ be two representations  of $G$ on $\mathcal{H}_1$ and $\mathcal{H}_2$ respectively. A linear map $M:\mathcal{H}_1\rightarrow\mathcal{H}_2$ commuting with $\rho_1$ and $\rho_2$, that is,
\begin{displaymath}
\rho_2(g)M=M\rho_1(g) \textrm{ for all } g\in G,
\end{displaymath}
is said to be $G$-linear. Representations $\rho_1$ and $\rho_2$ are said to be (linearly) equivalent, or more specifically $M$-equivalent, if there exist $G$-linear isomorphism $M:\mathcal{H}_1\rightarrow\mathcal{H}_2$.

A representation $\rho$ on $\mathcal{H}$ is said to be \emph{reducible} if there exist a complex, nontrivial $\rho(G)$-invariant subspace $\mathcal{H}'\subset\mathcal{H}$. Restriction of $\rho$ on $\mathcal{H}'$ yields a \emph{subrepresentation} of $\rho$ which is denoted by $\rho|_{\mathcal{H}'}$. The representation $\rho$ is called \emph{irreducible} if no $\rho(G)$-invariant subspaces exist. In this case for every nonzero $v\in\mathcal{H}$, an orbit $\Phi=(\rho(g)v)_{g\in G}$ is a frame in $\mathcal{H}$, if additionally $\rho$ is an unitary representation, then $\Phi$ is a tight frame.

\subsection{Symmetries of systems of subspaces}

Every $k$-dimensional subspace of $\mathcal{H}$ can be thought as a projective $(k-1)$-dimensional subspace of the projective space $\textrm{\textbf{P}}(\mathcal{H})$. All subspaces stay invariant under scalar transformations forming a center of $\GL(\mathcal{H})$. Thus, it is more natural to define a symmetry group of a system of subspaces as a subgroup of 
\begin{displaymath}
\PGL(\mathcal{H})=\GL(\mathcal{H})/Z(\GL(\mathcal{H})).
\end{displaymath} 
In the remainder of this section, we assume that all systems of subspaces span entire $\mathcal{H}$.

Two systems of subspaces $\mathcal{S}=(V_j)_{j\in S}$ and $\mathcal{S}'$ in $\mathcal{H}$ and $\mathcal{H}'$ respectively, are said to be \emph{projectively equivalent} if there exist a projective isomorphism $M$ between $\textrm{\textbf{P}}(\mathcal{H})$ and $\textrm{\textbf{P}}(\mathcal{H}')$ such that the sequence of subspaces $M\mathcal{S}=(MV_j)_{j\in S}$ can be reordered to $\mathcal{S}'$. The set of all maps $M\in\PGL(\mathcal{H})$ which permutes subspaces of $\mathcal{S}$ is called the symmetry group of $\mathcal{S}$ and is denoted by $\Sym(\mathcal{S})$. Symmetry groups of systems of subspaces are not always finite.
The \emph{stabilizer subgroup} of $V\in\mathcal{S}$ and \emph{$G$-transitivity} of $\mathcal{S}$ can be defined similarly as for frames.

Let $G\leq\PGL(\mathcal{H})$ be a finite group, we call any system of subspaces of the form $(gV)_{g\in G}$ where $V\subset\textbf{\textrm{P}}(\mathcal{H})$ is a subspace, an \emph{orbit of subspaces}. For any orbit of subspaces $\mathcal{S}$ we have $|\mathcal{S}|=|G|$, $G\leq\Sym(\mathcal{S})$ and $\mathcal{S}$ is $G$-transitive. If $V=hV$ for $V\in\mathcal{S}$ and $h\in G$, we say that $h$ \emph{stabilizes} $V$. The set of all elements $h\in G$ stabilizing $V$ form a \emph{stabilizer subgroup} $H_V$ of a subspace $V$ within $G$. Orbit of subspaces $\mathcal{S}$ contains $|H_V|$ copies of each subspace and there are in total $|G|/|H_V|$ distinct subspaces inside $\mathcal{S}$. In order to get rid of copies of subspaces inside $\mathcal{S}$ we can form a \emph{homogeneous} system of subspaces $\mathcal{S}'=(gV)_{gH_V\in G/H_V}$.

 A \emph{projective representation} of an abstract group $G$ on a Hilbert space $\mathcal{H}$ is a group homomorphism $\rho:G\rightarrow\PGL(\mathcal{H})$. A projective representation $\rho$ on $\mathcal{H}$ is said to be \emph{reducible} if there exist a nontrivial $\rho(G)$-invariant projective subspace $\mathcal{H}'\subset\textbf{\textrm{P}}(\mathcal{H})$. Otherwise projective representation is \emph{irreducible}. Projective representations $\rho_1$ and $\rho_2$ on Hilbert spaces $\mathcal{H}_1$ and $\mathcal{H}_2$ respectively are said to be projectively equivalent, if there exist a projective isomorphism $M:\textrm{\textbf{P}}(\mathcal{H}_1)\rightarrow\textrm{\textbf{P}}(\mathcal{H}_2)$ such that $M\rho_1(g)l=\rho_2(g)Ml$  for all $l\in\textrm{\textbf{P}}(\mathcal{H}_1)$ and $g\in G$.

The usefulnes of \emph{highly symmetric frames} defined in \cite{BW} stems from the observation that they often have small cardinalities and small angle sets. We will generalize this definition to systems of subspaces.

\begin{definition}[Highly symmetric systems of subspaces]\label{def1.1}
A system of distinct $k$-subspaces $\mathcal{S}$ ($k<d$) is highly symmetric if:
\begin{enumerate}
\item its symmetry group $\Sym(\mathcal{S})$ is irreducible and acts on subspaces of $\mathcal{S}$ transitively,
\item the stabilizer subgroup $H_V$ of any subspace $V\in\mathcal{S}$ is such that $H_V|_V$ is irreducible,
\item there exist a neighbourhood of $V\subset\textbf{P}(\mathcal{H})$ in which there are no other subrepresentations of $H_V$ projectively equivalent to $H_V|_V$.
\end{enumerate}
\end{definition}
\begin{remark} The last condition in Definition \ref{def1.1} is important. It isolates subspace $V\in\mathcal{S}$ from other $k$-subspaces of $\mathcal{H}$ stabilized by $H_V$. This ensures that $H_V$ can stabilize at most finitely many subspaces which belong to highly symmetric systems of subspaces.
\end{remark}
Highly symmetric line systems generalize line systems defined by the vectors of highly symmetric frames. Although a symmetry group of a highly symmetric line system $\mathcal{L}$ is not always finite, the next proposition guarantees that it is always possible to to pick finite $G\leq\Sym(\mathcal{L})$ which can be used to determine $\mathcal{L}$.

\begin{proposition}\label{prop:x}
For every  highly symmetric line system $\mathcal{L}$ there exist a finite irreducible $G\leq\Sym(\mathcal{L})$ with following properties:
\begin{enumerate}
\item $G$ acts on $\mathcal{L}$ transitively,
\item the stabilizer of any line $l\in\mathcal{L}$ within $G$ don't stabilize linewise any subspace that properly contains $l$.
\end{enumerate}
\end{proposition}
\begin{proof}
Take a subgroup $H\leq\Sym(\mathcal{L})$ containing all elements which fix all lines of $\mathcal{L}$ simultaneously. Each coset of $\Sym(\mathcal{L})/H$ is completely characterized by how it permutes lines within $\mathcal{L}$. Since lines of $\mathcal{L}$ can be permuted in finitely many ways, $\Sym(\mathcal{L})/H$ is finite. Moreover, $\Sym(\mathcal{L})/H$ has a group structure modulo $H$ as $H$ is a normal subgroup of $\Sym(\mathcal{L})$.

Let $d=\dim(\mathcal{H})$. Since $\mathcal{L}$ spans $\mathcal{H}$, pick lines $\{l_1,\ldots,l_d\}\subset\mathcal{L}$ spanning entire $\mathcal{H}$. By changing coordinates we can assume without loss of generality that these lines are orthogonal. Let $V=\{v_l:l\in\mathcal{L}\}$ be the set of unit vectors representing each line. Each vector $v_l$ has a unique orthonormal decomposition in terms of vectors $V_d=\{v_{l_1},\ldots,v_{l_d}\}$.

For any transformation $h\in H$ let $h'\in\GL(\mathcal{H})$ be one of its representatives. For any $v\in V\backslash V_d$ we have $h'v=\lambda_{h',v} v$. Now, since $v$ has a unique decomposition
\begin{displaymath}
v=\sum_{k=1}^da_kv_{l_k} \textrm{ where } a_k\in\mathbb{C} \textrm{ for } k=1,\ldots,d,
\end{displaymath} 
we notice that $h'v_{l_k}=\lambda_{h',v}v_{l_k}$ for all basis vectors $v_{l_k}$ such that $a_k\neq0$. Vectors $v_{l_k}$ with $a_k\neq0$ define a subspace $\mathcal{H}_v$ on which all transformations $h'$ act as scalars. By combining suitable subspaces $\mathcal{H}_v$ we obtain (unique) maximal subspaces $\mathcal{H}_i$ with orthonormal bases $\mathcal{B}_i\subset V_d$, $i=1,\ldots,p$, on which all transformations $h'$ act as scalars. Additionally we have $\mathcal{H}_i\bot\mathcal{H}_j$ when $i\neq j$, $\mathcal{H}_1\oplus,\ldots,\oplus\mathcal{H}_p=\mathcal{H}$, $\dim(\mathcal{H}_i)=d/p$ for all $i\in\{1,\ldots,p\}$ and every $v\in V$ belongs to exactly one subspace $\mathcal{H}_i$. For every two subspaces $\mathcal{H}_i,\mathcal{H}_j$, $i\neq j$ there exist at least one transform $h'\in\GL(\mathcal{H})$ representing some $h\in H$ such that
\begin{displaymath}
h'v_{\mathcal{H}_i}=\lambda_iv_{\mathcal{H}_i}\textrm{, }h'v_{\mathcal{H}_j}=\lambda_jv_{\mathcal{H}_j}\textrm{ for every }v_{\mathcal{H}_i}\in\mathcal{H}_i,v_{\mathcal{H}_j}\in\mathcal{H}_j\textrm{ and }\lambda_i\neq\lambda_j.
\end{displaymath}
On the other hand, any transformation $h'\in\GL(\mathcal{H})$ that acts on every subspace $\mathcal{H}_i$ as scalar represents of some $h\in H$. 

Representatives of conjugacy classes $\Sym(\mathcal{L})/H$ in addition to permuting lines of $\mathcal{L}$ also permute subspaces $\mathcal{H}_i$. Now, for any conjugacy class in $g\in\Sym(\mathcal{L})/H$ pick all representatives $g'\in\PGL(\mathcal{H})$ which can be represented by transforms $g''\in\GL(\mathcal{H})$ such that
\begin{displaymath}
\det(M_{i,j})=1\textrm{ for all pairs $\mathcal{H}_i$,$\mathcal{H}_j$ that $g''(\mathcal{H}_i)=\mathcal{H}_j$ and $M_{i,j}=[\langle g''(\psi),\phi\rangle]_{\psi\in\mathcal{B}_i,\phi\in\mathcal{B}_j}$.}
\end{displaymath}
For every conjugacy class $g$ there are in total $(d/p)^{p-1}$ such representatives $g'$. These representatives form a finite subgroup $G_1\leq\Sym(\mathcal{L})$ of order $(d/p)^{p-1}|\Sym(\mathcal{L})/H|$. Another finite subgroup $G_2\leq\PGL(\mathcal{L})$ of order $2^{p-1}$ which stabilizes all lines of $\mathcal{L}$ is represented by all linear transformations acting on every subspace $\mathcal{H}_i$, $i=1,\ldots,p$ as a scalars of the form $\pm1$. Together groups $G_1$ and $G_2$ generate a finite group $G\leq\Sym(\mathcal{L})$ of order at most $|G_1||G_2|$. We claim that $G$ fulfills all the criteria of Proposition \ref{prop:x}:
\begin{enumerate}
\item $G$ is irreducible,
\item $G$ acts transitively on $\mathcal{L}$,
\item the stabilizer of any line $l\in\mathcal{L}$ within $G$ don't stabilize linewise any subspace that properly contains $l$.
\end{enumerate}
The irreducibility of $G$ follows from the observation that $G$ is generated by its subgroups $G_1$ and $G_2$ while $\Sym(\mathcal{L})$ is generated by $G_1$ and $H$. Since (linear) representatives of elements of the group $G_2$ form a linear basis for (linear) representatives of $H$, we conclude that groups $G_2$ and $H$ share same invariant subspaces. Thus, since $\Sym(\mathcal{L})$ is irreducible, $G$ must also be. Similar argument can be also be used to deduce property (3). Transitivity of the action of $G$ on $\mathcal{L}$ follows from the fact that $G_1$ must act transitively on $\mathcal{L}$ since otherwise $\Sym(\mathcal{L})$ will fail to act transitively on $\mathcal{L}$.
\end{proof}

\subsection{Relation between symmetries of frames and line systems}

Group frames define group line systems and highly symmetric frames define highly symmetric line systems. Indeed, a linear representation of a finite group $\rho:G\rightarrow\GL(\mathcal{H})$ defines corresponding projective representation $\tilde{\rho}=\pi\circ\rho$ of $G$ where $\pi:\GL(\mathcal{H})\rightarrow\PGL(\mathcal{H})$ is a quotient map. Now, if $v\in\mathcal{H}$ and $\l=\{av:a\in\mathbb{C}\}$ then vectors of the group frame $\Phi=(\rho(g)v)_{g\in G}$ will represent lines of the line system $\mathcal{L}=(\tilde{\rho}(g)l)_{g\in G}$. Moreover, if corresponding homogeneous frame $\Phi'$ is highly symmetric, then so is corresponding homogeneous line system $\mathcal{L'}$. 

This relationship between group frames and line systems can be partially reversed. From the results of I. Schur we know that for any finite group $\tilde{G}$ there exist at least one central extension $G$ called a \emph{Shur cover} or a \emph{representation group} of $\tilde{G}$, such that every projective representation $\tilde{\rho}$ of $\tilde{G}$ can be \emph{lifted} to the ordinary representation $\rho$ of $G$ in the sense that
\begin{displaymath}
\pi(\rho(g))=\tilde{\rho}(\psi(g)) \textrm{ for all } g\in G,
\end{displaymath}
where $\psi:G\rightarrow\tilde{G}$ is the homomorphism associated with the group extension. 

A line system $\mathcal{L}=(\tilde{\rho}(\tilde{g})l)_{\tilde{g}\in\tilde{G}}$ defines a group frame $\Phi=(\rho(g)v)_{g\in G}$ where vector $v\in\mathcal{H}$ represents a line $l\in\textrm{P}(\mathcal{H})$. Each line $\tilde{\rho}(\tilde{g})l$ in $\mathcal{L}$ is represented by exactly $|\ker(\psi)|$ vectors $\rho(g_1)v$ in $\Phi$ where $g_1\in\psi^{-1}(\tilde{g})$. Let $\mathcal{L}'$ be a homogeneous line system obtained from $\mathcal{L}$ and $\Phi'$ a homogeneous frame obtained from $\Phi$. Now, the high symmetricity of $\mathcal{L}'$ does not necessarily translate into a high symmetricity of $\Phi'$ since linear transforms in $\Sym(\Phi')$ which stabilize line $l\in\mathcal{L}'$ setwise does not necessarily stabilize vectors representing it.

In order to restore correspondence between highly symmetric line systems and highly symmetric frames we need to redefine the latter by utilizing line stabilizer subgroups instead of vector stabilizer subgroups.
\begin{definition}[Generalized highly symmetric frames]\label{def2}
A finite frame $\Phi$ of distinct vectors is generalized highly symmetric if the action of its symmetry group $\Sym(\Phi)$ is irreducible, transitive and stabilizer of any line represented by a frame vector doesn't simultaneously satabilize all lines on any subspace of $\mathcal{H}$ properly containing $l$.
\end{definition}

From Definition \ref{def2} follows that any highly symmetric frame is also generalized highly symmetric. Additionally, any highly symmetric line system $\mathcal{L}'$ can be "lifted" to generalized highly symmetric frame $\Phi'$ in following steps:
\begin{enumerate}
\item Use  (for example) Proposition \ref{prop:x} to determine some finite $\tilde{G}\leq\Sym(\mathcal{L}')$ which determines lines of $\mathcal{L}$ and lift it to ordinary representation $G$.
\item Construct group frame $\Phi=\{gv\}_{g\in G}$ where vector $v$ represents some line $l\in\mathcal{L}'$. Obtain generalized highly symmetric frame $\Phi'$ by homogenizing $\Phi$.
\end{enumerate}
Generalized highly symmetric frame $\Phi'$ represents each line $l\in\mathcal{L}'$ by vectors of the form $e^{2k\pi i/n}v\in l$, $k\in\{0,\ldots,n-1\}$ forming a complex regular $n$-gons. Indeed, $H_l|_l$ (where $H_l\leq\Sym(\Phi')$) is isomorphic to some homomorphism $h:H_l\rightarrow\mathbb{C}\cong l$. An $h$-image of $H_l$ is a finite multiplicative subgroup of $\mathbb{C}$ which must be generated by a single $n$-th root of unity for some $n\in\mathbb{N}$ and form a complex regular $n$-gon. 

\section{Line systems related to irreducible representations of finite groups}

We are interested in highly symmetric line systems, which,  according to Section 2, can always be obtained from generalized highly symmetric frames. There seems to be several methods to construct such frames from representations of finite groups. For example, in \cite{BW} an approach utilizing a linear algebra and explicit linear representations of finite groups was used in construction of highly symmetric frames. In \cite{IJM}, on the other hand, an approach involving spherical functions was used in order to accomplish essentially the same goal. We are going to use a slight extension of the second approach, mainly for the reason that it does not require explicit linear representations of finite groups. In particular, only character tables and subgroup structures of finite groups are needed in construction of generalized highly symmetric frames associated with irreducible representations. Both, characters and subgroup structure can be computed from permutation representations of finite groups which are often available in common computer algebra packages such as Magma and GAP \cite{BCP,G}.

In the first part of this section we continue intoducing representation theory of finite groups needed in construction of group frames. As before, see \cite{CR,W} for more details.
\subsection{Representations of finite groups} 
Let $\rho$ be a representation of the finite group $G$ on the Hilbert space $\mathcal{H}$. The function $\chi_\rho:G\rightarrow\mathbb{C}$ defined by $\chi_\rho(g)=\Tr(\rho(g))$ where $\Tr()$ is a trace of a linear map is the  \emph{character} associated with the representation $\rho$. A character $\chi_\rho$ is called irreducible if $\rho$ is an irreducible representation. If $\rho$ is one-dimensional representation, then $\chi_\rho$ is called a \emph{linear character}. Characters have the following useful properties:
\begin{itemize}
\item characters are \emph{class functions}, i.e., are constant on a given conjugacy class,
\item the set of all irreducible characters $\{\chi_i:i=1,\ldots,N\}$ forms an orthogonal basis of the class functions, i.e.,
\end{itemize}
\begin{displaymath}
\sum_{g\in G}\chi_i(g)\overline{\chi_j(g)}=\begin{cases}
|G|,&\text{if } i=j \\
0,&\text{otherwise,}
\end{cases}
\end{displaymath}
\begin{itemize}
\item representations of $G$ have the same character if and only if they are linearly equivalent,
\item for any character $\chi_\rho$ we have $\chi_\rho(1)=\Dim(\mathcal{H})$ and $\chi_\rho(g^{-1})=\overline{\chi_\rho(g)}$ for all $g\in G$.
\end{itemize}
A restrictions of a representation $\rho$  and a corresponding character $\chi_\rho$ to a subgroup $H\leq G$ are denoted by $\rho\downarrow_H^G$ and $\chi_\rho\downarrow^G_H$ respectively. According to the Maschke's theorem, $\mathcal{H}$ can be decomposed into a direct sum of irreducible subrepresentations. This decomposition is not unique but the number of irreducible constituents of each type is independent of the choice of decomposition. Nevertheless, if all irreducible subrepresentations of the same type are combined into a direct sum, resulting \emph{isotypic} subrepresentation is unique. A projection onto isotypic subspace associated with an irreducible character $\chi$ is given by
\begin{equation}\label{eq:2}
\Pi_{\chi}=\frac{\chi(1)}{|G|}\sum_{g\in G}\overline{\chi}(g)\rho(g).
\end{equation}
Similarly, character $\chi$ can be uniquely decomposed into irreducible components
\begin{displaymath}
\chi=\sum_{i=1}^N a_i\chi_i,
\end{displaymath}
where integers $a_i$ are called \emph{multiplicities} of each irreducible representation in $\chi$.

Let $G$ be a finite group. The convolution of two elements $f,h\in L^2(G)$ is defined by
\begin{displaymath}
f*h(s)=\sum_{g\in G}f(g)h(g^{-1}s)
\end{displaymath}
which turns $L^2(G)$ into an associative group convolution algebra. The \emph{(left) regular representation} of $G$ on $L^2(G)$ is a homomorphism
\begin{displaymath}
\rho:G\rightarrow \GL(L^2(G))\textrm{ such that } \rho(g)v=\delta_{g}*v\textrm{ for all } g\in G\textrm{, } v\in L^2(G),
\end{displaymath}
where
\begin{displaymath}
\delta_g(s)=\begin{cases}
1,&\text{if } s=g \\
0,&\text{otherwise.}
\end{cases}
\end{displaymath}
In the remainder of this section, depending on the context, we will assume that $L^2(G)$ is either a (left) regular representation, a group convolution algebra or both. Note, that by associativity of the group convolution, the $|G|$-dimensional vector space of operators
\begin{equation}\label{eq:x}
 A_h:L^2(G)\rightarrow L^2(G)\textrm{ , } \textrm{ defined by }A_hf=f*h\textrm{ for all }f\in L^2(G)
\end{equation}
is commuting with all elements of the (left) regular representation. Next we will need Schur's lemma, an important statement about irreducible representations.

\begin{lemma}[Schur's lemma] Let $\rho_1$ and $\rho_2$ be two complex irreducible representations of a finite group $G$ on Hilbert spaces $\mathcal{H}_1$ and $\mathcal{H}_2$ respectively. If $\rho_1$ and $\rho_2$ are linearly equivalent, then the space of $G$-linear maps between $\mathcal{H}_1$ and $\mathcal{H}_2$ is one-dimensional. Otherwise, any $G$-linear map $M:\mathcal{H}_1\rightarrow\mathcal{H}_2$ is 0.
\end{lemma}

From Schur's lemma follows that each operator $A_h$ can be written as a sum of $G$-linear isomorphisms between irreducible subspaces of $L^2(G)$.

Let $\rho$ be the (left) regular representation and $\chi$ some irreducible character of $G$. According to \eqref{eq:2} the projection $\Pi_\chi$ onto the $\chi$-isotypic subspace of $L^2(G)$ is defined by
\begin{displaymath}
\Pi_{\chi}(f)=(\frac{\chi(1)}{|G|}\sum_{g\in G}\overline{\chi}(g)\rho(g))(f)=\frac{\chi(1)}{|G|}\cdot\sum_{g\in G}\overline{\chi}(g)\delta_{g}*f=\frac{\chi(1)}{|G|}\cdot\overline{\chi}*f
\end{displaymath}
for all $f\in L^2(G)$. We observe that for any $f\in L^2(G)$
\begin{equation}\label{eq:3}
(\overline{\chi}*f)(s)=\sum_{g\in G}\overline{\chi}(g)f(g^{-1}s)=\sum_{g\in G}\overline{\chi}(tg^{-1})f(g)=\sum_{g\in G}\overline{\chi}(g^{-1}t)f(g)=(f*\overline{\chi})(s)
\end{equation}
by using the fact that group characters are class functions. An isotypic subspace associated with an irreducible character $\chi$ of degree $d$ is clearly $d^2$-dimensional and thus is a direct sum of $d$ irreducible subspaces $V_{\chi,k}$ where $k=1,\ldots,d$. From this and the Schur's lemma follows that \eqref{eq:x} defines all $G$-linear operators from $L^2(G)$ to $L^2(G)$. 

\subsection{Frames from projections onto irreducible subspaces of the regular representation}

The regular representation has an important connection with Gram matrices of group tight frames. Let $\rho'$ be a $d$-dimensional representation of a finite group $G$ on a Hilbert space $\mathcal{H}$ and $v\in \mathcal{H}$ such that $\Phi=(\rho'(g)v)_{g\in G}$ is a tight frame. Now, since $\rho'$ must be a unitary representation, the Gram matrix of $\Phi$ can be written in the following form
\begin{displaymath}
G_\Phi=[\langle\rho'(g_2)v,\rho'(g_1)v\rangle]_{g_1,g_2\in G}=[\langle v,\rho'(g_2^{-1}g_1)v\rangle]_{g_1,g_2\in G}.
\end{displaymath}
This can be interpreted as a group convolution matrix associated with the function 
\begin{equation}\label{eq:p}
h:L^2(G)\rightarrow\mathbb{C}\textrm{ such that }h(g)=\langle v,\rho'(g)v\rangle\textrm{ for all }g\in G
\end{equation}
in the sense that $G_\Phi v'=v'*h$ for all $v'\in L^2(G)$. Since $\Phi$ is a tight frame, $G_\Phi$ is a projection matrix (modulo scalar factor) onto some subspace $V\subset L^2(G)$. This subspace stays invariant under (left) regular representation since $G_\Phi\rho(g)=\rho(g)G_\Phi$ for all $g\in G$. Moreover, columns of the Gram matrix $G_\Phi$ define an unitarily equivalent group frame in $V$, equivalence can be defined as linear extension of following function  $M:\Phi\rightarrow V$
\begin{displaymath}
M\rho'(g)v=\frac{1}{\rVert v\lVert_2}\sqrt{\frac{d}{|G|}}G_\Phi[:,g]\textrm{ for all }g\in G.
\end{displaymath}
Similarly, column vectors of any group convolution projection matrices produce group tight frames in $\rho$-invariant subspaces of $L^2(G)$.

Take a representation $\rho':G\rightarrow\U(\mathcal{H})$ of a finite group $G$ and choose a subgroup $H\leq G$. By restricting $\rho'$ to $H$ it is possible to decompose $\mathcal{H}$ into $H$-irreducible subspaces. Let $V$ be one of such subspaces of $\mathcal{H}$ associated with an irreducible character $\chi_V$ of $\rho'\downarrow^{G}_{H}$ and $\Phi_V$ be a $\rho'\downarrow^{G}_H$-frame with an initial vector $v\in V$. The Gram matrix of $\Phi_V$ defines a projection matrix $\Pi_{\Phi_V}$ onto some $\chi_V$-irreducible subspace of $L^2(H)$ defined by the function
\begin{displaymath}
h_V:L^2(H)\rightarrow\mathbb{C}\textrm{ such that }h_V(g)=\frac{\Dim(V)}{\langle v,v\rangle|H|}\langle v,\rho'(g)v\rangle\textrm{ for all }g\in H.
\end{displaymath}
The matrix $\Pi_{\Phi_V}$ can be extended to a projection matrix $\Pi_{\Phi_V}'$ onto some invariant subspace of $L^2(G)$ defined by the function $h:L^2(G)\rightarrow\mathbb{C}$ such that
\begin{equation}\label{eq:ext}
h(s)=\begin{cases}
h_V(s),&\text{if } s\in H \\
0,&\text{otherwise.}
\end{cases}
\end{equation}
In the proposition below we are going to "spin" $\rho'\downarrow^G_H$-frames lying inside $H$-irreducible subspaces of $\mathcal{H}$ by elements of $\rho'(G)$. Obtained $G$-frames will be associated with highly symmetric systems of subspaces.
\begin{proposition}\label{prop:1}
Let $G$ be a finite group and $\rho':G\rightarrow\U(\mathcal{H})$ an irreducible $d$-dimensional representation with character $\chi$. 
Let $H\leq G$, $\nu:H\rightarrow\mathbb{C}$ an irreducible character and $\Pi_{\Phi_1}$ a projection operator onto $\nu$-irreducible subspace of $L^2(H)$. Then there exist the unique subspace $V\subset\mathcal{H}$ of dimension $\nu(1)$ and a vector $v\in V$ unique modulo scalar factor, such that
\begin{equation}\label{eq:5}
\langle v,\rho'(s)v\rangle=\frac{\langle v,v\rangle|H|}{\Dim(V)}\Pi_{\Phi_1}[h,1] \textrm{ for all } s\in H
\end{equation}
if and only if
\begin{equation}\label{eq:6}
\sum_{h\in H}\chi(s)\overline{\nu(s)}=|H|.
\end{equation}
The corresponding group tight frame $\Phi=(\rho'(g)v)_{g\in G}$ has the following Gram matrix
\begin{equation}\label{eq:7}
G_\Phi=\langle v,v\rangle\frac{|G|}{d}\Pi_\chi\Pi_{\Phi_1}'.
\end{equation}
\end{proposition}
\begin{proof}
If \eqref{eq:6} holds, then the multiplicity of $\nu$-irreducible subrepresentation inside $\rho'\downarrow^G_H$ is one, and there is exactly one $\nu$-irreducible subspace $V\subset\mathcal{H}$. Column vectors of the projection matrix $\Pi_{\Phi_1}$ define a tight frame inside some $\nu$-irreducible subspace $V'\subset L^2(H)$. Now, uniqueness modulo scalar factor of the vector $v$ follows (by Schurs lemma) from one-dimensionality of $H$-unitary maps $M:V'\rightarrow V$.

If \eqref{eq:6} does not hold, then there are two possibilities. Either
 \begin{displaymath}
\sum_{s\in H}\chi(s)\overline{\nu(s)}=0.
\end{displaymath}
and there are no  $\nu$-irreducible subrepresentations inside $\rho'\downarrow^G_H$ meaning that \eqref{eq:5} cannot hold for any vector $v\in\mathcal{H}$. Or
\begin{displaymath}
\sum_{s\in H}\chi(s)\overline{\nu(s)}=k|H| \textrm{ with } k>1
\end{displaymath}
and $\nu$-isotypic subrepresentation of $\rho'\downarrow^G_H$ has dimension greater than $\nu(1)$. In any $\nu$-irreducible subspace exist a vector $v$ such that \eqref{eq:5} holds and uniqueness fails in this case.

The matrix $G_\Phi$ is characterized by the following properties:
\begin{enumerate}
\item Equation \eqref{eq:5} holds,
\item modulo a constant, $G_\Phi$ is a projection matrix onto $\chi$-irreducible subspace of $L^2(G)$.
\end{enumerate}
All $\rho$-invariant subspaces of $L^2(G)$ which satisfy the first condition can be found inside $\Pi_{\Phi_1}'(L^2(G))$. On the other hand, all $\chi$-irreducible subspaces can be found inside $\Pi_{\chi}(L^2(G))$. Now,  $\Pi_\chi$ and $\Pi_{\Phi_1}'$ are commuting because of \eqref{eq:3}. The product $\Pi_\chi\Pi_{\Phi_1}'$ is again a projection matrix onto some invariant subspace and thus defines a tight frame. The squared norm of the vectors of this frame can be computed from diagonal entries of $\Pi_\chi\Pi_{\Phi_1}'$. It is equal to
\begin{displaymath}
\sum_{s\in H}\frac{d}{|G|}\chi(s)h(s)=\sum_{s\in H}\frac{d}{|G|}\nu(s)h(s)=\frac{d}{|G|}
\end{displaymath}
since $\nu$ occurs in $\chi\downarrow_H^G$ with multiplicity 1. Thus, $\Pi_\chi\Pi_{\Phi_1}'$ projects onto unique $\chi$-irreducible subspace of $L^2(G)$ and the matrix $\langle v,v\rangle\frac{|G|}{n}\Pi_\chi\Pi_{\Phi_1}'$  satisfy both conditions (1) and (2).
\end{proof}
Below are some comments related to Proposition \ref{prop:1}.
\begin{itemize}
\item Frame vectors of $\Phi$ labelled by the elements of the same $G/H$-coset form a tight frame in some subspace isomorphic to $V$. Distinct subspaces of this type form a highly symmetric system of subspaces. If $\nu$ is a linear character, frame $\Phi$ defines a highly symmetric line system and a homogenization of $\Phi$ produces a generalized highly symmetric frame $\Phi'$.
\item According to \eqref{eq:7} the first row of $G_\Phi$ is given by
\end{itemize}
\begin{equation}\label{eq:sph}
G_\Phi[1_G,g]=\frac{\langle v,v\rangle}{|H|}\sum_{s\in H}\chi(sg^{-1})\overline{\nu(s)} \textrm{ for all } g\in G,
\end{equation}
\begin{itemize}
\item[$$]containing all the necessary information since $G_\Phi$ is a group convolution matrix. If $\nu$ is a linear character, only the values on the representatives of $(H,H)$-double cosets require explicit computation. In this case equation \eqref{eq:sph} (modulo factor $\langle v,v\rangle$) defines a \emph{twisted sphercal function}. The twisted spherical function has constant modulus on each $(H,H)$-double coset of $G$. The number of the frame angles is thus smaller than the number of different $(H,H)$-double cosets in $G$.
\item If $\sum_{h\in H}\chi(h)\overline{\nu(h)}=k|H|$ with $k>1$ then \eqref{eq:7} still produces Gram matrices of group tight frames. In this case matrix $\Pi_\chi\Pi_\nu$ projects onto reducible subspace of $L^2(G)$.
\item If $\sum_{h\in H}\chi(h)\overline{\nu(h)}=k|H|$ with $k>1$,  there exist a $(k-1)$-dimensional family of $\rho'(G)$-frames in $\mathcal{H}$ such that \eqref{eq:5} holds for the initial vector of the frame (i.e., vector indexed by the identity $1_G$).
\end{itemize}

\subsection{Angular relations between different group frames} We already constructed Gram matrices of generalized highly symmetric frames, next we are going to acquire additional information regarding angular relations between these frames. 

As before, let $\rho':G\rightarrow\U(\mathcal{H})$ be an irreducible $d$-dimensional representation of a finite group $G$ with character $\chi$. Let $v_1,v_2\in\mathcal{H}$ and $\Phi_1=(\rho'(g)v_1)_{g\in G}$ and  $\Phi_2=(\rho'(g)v_2)_{g\in G}$ be two group frames in $\mathcal{H}$ with Gram matrices $G_{\Phi_1}$ and $G_{\Phi_2}$ respectively. Let $V_1\subset L^2(G)$ be the $\chi$-invariant subspace defined by $G_{\Phi_1}$ and  $V_2\subset L^2(G)$ be the $\chi$-invariant subspace defined by $G_{\Phi_2}$. According to Schur's lemma there exist following $G$-unitary maps
\begin{equation}\label{eq:G}
M_1:\mathcal{H}\rightarrow V_1\textrm{, }M_2:\mathcal{H}\rightarrow V_2\textrm{ and }M:V_1\rightarrow V_2,
\end{equation}
unique modulo phases which can be choosen in such a way that $M_1\rho'(g_1)v_1$ and $M_2\rho'(g_2)v_2$ are
\begin{displaymath}
\frac{1}{\lVert v_1\rVert_2}\sqrt{d/|G|}G_{\Phi_1}[:,g_1]\textrm{ and }\frac{1}{\lVert v_2\rVert_2}\sqrt{d/|G|}G_{\Phi_2}[:,g_2]
\end{displaymath}
respectively for all $g_1,g_2\in G$. According to Schur's lemma we know that
\begin{displaymath}
F:=M_2^{-1}\circ M\circ M_1=c\cdot\textrm{Id}_\mathcal{H}
\end{displaymath}
where $c$ is a phase. Now, for $\rho'(g_1)v_1$ and  $\rho'(g_2)v_2$ we have
\begin{displaymath}
c\langle \rho'(g_1)v_1,\rho'(g_2)v_2\rangle=\langle F\rho'(g_1)v_1,\rho'(g_2)v_2\rangle=\langle M\circ M_1\rho'(g_1)v_1,M_2\rho'(g_2)v_2\rangle.
\end{displaymath}
Next, observe that Id$_{L^2(G)}$ is (trivially) a $G$-unitary projection onto $L^2(G)$. By decomposing $L^2(G)$ in two ways $L^2(G)=V_i\oplus V_i^\bot$, we obtain
\begin{displaymath}
\textrm{Id}_{L^2(G)}=\Pi_{V_i}+\Pi_{V_i^\bot} \textrm{, }i=1,2.
\end{displaymath}
Since rows of $M$ lay entirely in $\overline{V}_1$ while columns lay entirely in $V_2$ we have
\begin{displaymath}
M[g_2,g_1]=\langle M\delta_{g_1},\delta_{g_2}\rangle=\langle M(\Pi_{V_1}[:,g_1]+\Pi_{V_1^\bot}[:,g_1]),\delta_{g_2}\rangle=
\end{displaymath}
\begin{displaymath}
=\langle\Pi_{V_1}[:,g_1],M^*(\Pi_{V_1}[:,g_2]+\Pi_{V_1^\bot}[:,g_2])\rangle=\frac{cd}{\lVert v_1\rVert_2\lVert v_2\rVert_2|G|}\langle \rho'(g_1)v_1,\rho'(g_2)v_2\rangle.
\end{displaymath}
Thus, the following result is proved.
\begin{proposition}
Let $M$ be as in \eqref{eq:G} then
\begin{equation}\label{eq:I}
M_{\Phi_1\Phi_2}[g_2,g_1]:=
\frac{\lVert v_1\rVert_2\lVert v_2\rVert_2|G|}{cd}M[g_2,g_1]=\langle \rho'(g_1)v_1,\rho'(g_2)v_2\rangle\textrm{ for all }g_1,g_2\in G.
\end{equation}
\end{proposition}
Matrix $M$ can be constructed explicitly modulo scalar factor.  We have
\begin{displaymath}
\Pi_{V_i}f=\frac{n}{\langle v_i,v_i\rangle|G|}\cdot f*h_i\textrm{ for all }f\in L^2(G)
\end{displaymath}
where $h_i$ are defined as in \eqref{eq:p} from Gram matrices $G_{\Phi_i}$, $i=1,2$.
Let $\chi_a:G\rightarrow\mathbb{C}$ where $a\in G$, be defined by $\chi_a(g)=\chi(ga)$ for all $g\in G$ and $A_{\chi,a}:L^2(G)\rightarrow L^2(G)$ be defined by
\begin{displaymath}
A_{\chi,a}f=f*\overline{\chi_a}\textrm{ for all } f\in L^2(G).
\end{displaymath}
\begin{proposition}For all $f\in L^2(G)$
\begin{equation}\label{eq:8}
\lambda_a Mf=\Pi_{V_2}A_{\chi,a}\Pi_{V_1}f
\end{equation}
where $\lambda_a\in\mathbb{C}$ is not identically zero for all $a\in G$.
\end{proposition}
\begin{proof}
Let $V_\chi$ be the $\chi$-isotypic component of $L^2(G)$. We perform two different decompositions of $V_\chi$ into $\chi$-irreducible orthogonal subspaces:
\begin{displaymath}
V_\chi=V_1\oplus W_2\oplus\ldots\oplus W_n \textrm{ and } V_\chi=V_2\oplus W_2'\oplus\ldots\oplus W_n'.
\end{displaymath}
Since function $\chi_a\in V_\chi$, according to Schur's lemma $A_{\chi,a}$ can be expressed in the form
\begin{displaymath}
A_{\chi,a}=\lambda_{a} M+\sum_{j=2}^N\lambda_{a,1,j}L_{V_1,W_j'}+\sum_{i,j=2}^{N}\lambda_{a,i,j}L_{W_i,W_j'}+\sum_{i=2}^N\lambda_{a,i,1}L_{W_i,V_2}
\end{displaymath}
where $\lambda_a,\lambda_{a,i,j}\in\mathbb{C}$ are constants and $L_{X,Y}:X\rightarrow Y$ are $G$-linear isomorphisms. Now, 
\begin{displaymath}
(\sum_{a,i,j=2}^{N}\lambda_{i,j}L_{W_i,W_j'}+\sum_{i=2}^N\lambda_{a,i,1}L_{W_i,V_2})\Pi_{V_1}=0
\end{displaymath}
and
\begin{displaymath}
(\lambda_{a} M+\sum_{j=2}^N\lambda_{a,1,j}L_{V_1,W_j'})\Pi_{V_1}=\lambda_{a} M+\sum_{j=2}^N\lambda_{a,1,j}L_{V_1,W_j'}.
\end{displaymath}
Similarly,
\begin{displaymath}
\Pi_{V_2}(\sum_{j=2}^N\lambda_{a,1,j}L_{V_1,W_j'}+\sum_{a,i,j=2}^{N}\lambda_{i,j}L_{W_i,W_j'})=0
\end{displaymath}
and
\begin{displaymath}
\Pi_{V_2}(\lambda_{a} M+\sum_{i=2}^N\lambda_{a,i,1}L_{W_i,V_2})=\lambda_{a} M+\sum_{i=2}^N\lambda_{a,i,1}L_{W_i,V_2}.
\end{displaymath}
By combining four previous equations together we get \eqref{eq:8}. Since $M$ is $G$-linear, it must be of the form \eqref{eq:x}. On the other hand, all rows of $M$ lie in $\overline{V_1}\subset \overline{V_\chi}$. Since functions $\chi_a$ form an overdefined basis of $V_\chi$, we can express $M$ in the form
\begin{displaymath}
Mf=f*\sum_{a\in G}b_a\overline{\chi_a}\textrm{ for all } f\in L^2(G)
\end{displaymath}
with some constants $b_a\in\mathbb{C}$. This proves that $\lambda_a$ is not zero for all $a\in G$.
\end{proof}
We are especially interested in the situation where both $G$-frames $\Phi_1$ and $\Phi_2$ are related to generalized highly symmetric frames in the sense of Proposition \ref{prop:1}. In this case we can obtain simpler expression for $M$.

\begin{proposition}\label{prop:4} Let $\rho':G\rightarrow\U(\mathcal{H})$ be a $d$-dimensional irreducible representation of a finite group $G$ with a character $\chi$. Let $H_1,H_2\leq G$ and $\nu_1,\nu_2$ be linear characters of $H_1$ and $H_2$ respectively such that \eqref{eq:6} holds in both cases. Let $\Pi_{\nu_i}$, $i=1,2$ be projection operators onto $\nu$-irreducible subspaces of $L^2(H_i)$, and $\Pi_{\nu_i}'$ their extensions as in \eqref{eq:ext}. Let $\Phi_1$ and $\Phi_2$ be group frames with Gram matrices
\begin{displaymath}
G_{\Phi_i}=\langle v_i,v_i\rangle\frac{|G|}{n}\Pi_\chi\Pi_{\nu_i}'\textrm{, }i=1,2.
\end{displaymath}
Then the $G$-linear map $M$ as in \eqref{eq:G} is defined by
\begin{equation}\label{eq:9}
\lambda_aM[s,g]=\frac{1}{|H_1||H_2|}\sum_{h_1\in H_1}\sum_{h_2\in H_2}\overline{\nu_2(h_2)}\chi(a^{-1}h_2s^{-1}gh_1)\overline{\nu_1(h_1)} \textrm{ for all } s,g\in G
\end{equation}
where $\lambda_a\in\mathbb{C}$ depends only on $a\in G$.
\end{proposition}
\begin{proof} By \eqref{eq:8} we obtain 
\begin{displaymath}
\lambda_a Mf=\Pi_\chi\Pi_{\nu_2}'A_{\chi,a}\Pi_\chi\Pi_{\nu_1}'f \textrm{ for all }f\in L^2(G).
\end{displaymath}
We can omit operators $\Pi_\chi$ since they both are identities  on $V_\chi$. We get
\begin{equation}\label{eq:new}
\lambda_a Mf=\Pi_{\nu_2}'A_{\chi,a}\Pi_{\nu_1}'f=f*(\phi_{\nu_1}*\overline{\chi_a}*\phi_{\nu_2}) \textrm{ for all }f\in L^2(G)
\end{equation}
where for $i=1,2$
\begin{displaymath}
\phi_{\nu_i}(s)=\begin{cases}
\nu_i(s)/|H_i|,&\text{if } s\in H_i \\
0,&\text{otherwise.}
\end{cases}
\end{displaymath}
Expression in the parentesis in \eqref{eq:new} can be denoted my $m_a$ and computed explicitly
\begin{equation}\label{eq:I1}
m_a(g)=\sum_{h_1\in G}\phi_{\nu_1}(h_1)\sum_{h_2\in G}\overline{\chi(h_2a)}\phi_{\nu_2}(h_2^{-1}h_1^{-1}g)\textrm{ for all }g\in G.
\end{equation}
By substituting $h_2\rightarrow h_1^{-1}rh_2^{-1}$ we get
\begin{displaymath}
m_a(g)=\sum_{h_1\in G}\phi_{\nu_1}(h_1)\sum_{h_2\in G}\overline{\chi(h_1^{-1}gh_2^{-1}a)}\phi_{\nu_2}(h_2)=\frac{1}{|H_1||H_2|}\sum_{h_1\in H_1}\sum_{h_2\in H_2}\overline{\nu_1(h_1)}\chi(a^{-1}h_2g^{-1}h_1)\overline{\nu_2(h_2)}
\end{displaymath}
for all $g\in G$, which is exactly what we needed to prove.
\end{proof}
Some comments regarding Proposition \ref{prop:4}.
\begin{itemize}
\item Unless $\lambda_a=0$, by renormalizing $\lambda_aM$ it is possible to reconstruct $M_{\Phi_1\Phi_2}$ modulo phase. If for some $a\in G$  $\lambda_a=0$, try different  $a$ in \eqref{eq:9} until $\lambda_a\neq0$. 
\item The function $m_a$ in the proof of Proposition \ref{prop:4} has the following property 
\end{itemize}
\begin{displaymath}
m_a(t_1gt_2)=\overline{\nu_1(t_1)}\overline{\nu_2(t_2)}m_a(g)\textrm{ for all }t_1\in H_1,t_2\in H_2\textrm{ and }g\in G.
\end{displaymath}
\begin{itemize}
\item[$$]I.e., has constant modulus on $(H_1,H_2)$-double cosets of $G$. The number of distinct angles between vectors of $\Phi_1$ and $\Phi_2$ is therefore limited by the number of $(H_1,H_2)$-double cosets of $G$.
\item Column vectors of $\lambda_a M$ define a $G$-frame $\tilde{\Phi}_1$ in $V_2$ which is modulo scalar factor unitarily equivalent to $\Phi_1$ while column vectors of $G_{\Phi_2}$ define $G$-frame $\tilde{\Phi}_2$ which is modulo scalar factor unitarily equivalent to $\Phi_2$.
\end{itemize}

\section{Examples}
In this section we give some examples of frames which can be constructed by using propositions \ref{prop:1} and \ref{prop:4}. The first example clarifies notations used in the remainder of this section. In the second example all highly symmetric line systems related to $\textrm{W}(\textrm{H}_4)$ are classified. The third example shows how highly symmetric line systems can define other highly symmetric systems of subspaces. In remaining examples we construct few nice spherical codes/systems of lines from smaller (generalized) highly symmetric frames. The notation $\Phi_1\cup\Phi_2$ where $\Phi_1$ and $\Phi_2$ are frames in the same Hilbert space will be reserved to any frame $\Phi$ containing vectors of both frames $\Phi_1$ and $\Phi_2$ in any order.

\subsection{Notations} Let $G\cong A_5$ and $\rho:G\rightarrow\OO(\mathbb{R}^3)$ one of its irreducible 3-dimensional representations. There is a single conjugacy class of subgroups isomorphic to a cyclic group with 5 elements. By taking one such subgroup denoted by $H$, compute twisted spherical functions associated with its linear characters. Group $H$ has 5 different linear characters, one trivial and 4 of order 5 coming in conjugated pairs. There is a spherical function associated with the trivial character and two twisted spherical functions associated with one of the pairs of conjugated characters denoted by $\nu_1$ and $\nu_2$ respectively. Group $G$ has 4 $(H,H)$-double cosets of sizes 5, 25, 25 and 5. By choosing representatives of each double coset it is possible to compute value of each (twisted) spherical function on each representative of double coset:
\begin{align*}
 &[1,5][\frac{1}{\sqrt{5}},25][-\frac{1}{\sqrt{5}},25][-1,5],\\
 &[1,5][\frac{5+\sqrt{5}}{10}e^{\frac{3\pi i}{5}},25][-\frac{5-\sqrt{5}}{10},25][0,5],\\
 &[1,5][\frac{5+\sqrt{5}}{10}e^{\frac{-3\pi i}{5}},25][-\frac{5-\sqrt{5}}{10},25][0,5].
\end{align*}
Here, the first line is related to a  spherical function associated with the trivial linear character. The line stabilizer of a line system defined by this function is of cardinality $5+5=10$. The corresponding highly symmetric line system has $|A_5|/10=6$ lines and is $\{1/\sqrt{5}\}$-angular, each line representing an antipodal pair of vertices of the regular icosahedron. Other (twisted) spherical functions are associated with the same $\{\frac{5+\sqrt{5}}{10},\frac{5-\sqrt{5}}{10},0\}$-angular highly symmetric line system $\mathcal{L}$ of cardinality 10. Let $\Phi_1$ be a (generalized) highly symmetric frame associated with $\nu_1$ and $\Phi_2$ (generalized) highly symmetric frame associated with $\nu_2$. The first row of matrix $M_{\Phi_1\Phi_2}$ of \eqref{eq:I} is described (modulo phase $c_{1,2}$) by following values on representatives of $(H,H)$-double cosets
\begin{displaymath}
c_{1,2}[0,1][\frac{5+\sqrt{5}}{10}e^{\frac{9\pi i}{5}},5][-\frac{5+\sqrt{5}}{10},5][1,1].
\end{displaymath}
Notice that some inner products between vectors of $\Phi_1$ and $\Phi_2$ have amplitude 1 and these frames define the same line system $\mathcal{L}$. Turns out that pairs $(H,\nu_1)$ and $(H,\nu_2)$ stabilize different lines of $\mathcal{L}$. In fact, conjugacy class of subgroups containing $H$ has only 5 members, each member stabilizing exactly $2$ lines of $\mathcal{L}$ through different linear characters of order 5. 

\subsection{Classification of highly symmetric frames}
Highly symmetric frames defined by 34 exceptional irreducible complex reflection groups in the Shephard–Todd classification were classified in \cite {BW,HW}. It is possible to extend this classification to generalized highly symmetric frames by using Proposition \ref{prop:1}. Table below lists all generalized highly symmetric frames defined by the real reflection group $\textrm{W}(\textrm{H}_4)=\textrm{ST(30)}$ (the automorphism group of the 600-cell). Frames marked by $*$ are not highly symmetric and are therefore missing from the classification in \cite{HW}. New frames are complex despite $\textrm{W}(\textrm{H}_4)$ being real.
\begin{table}[htbp]%
	\centering
	\caption{Parameters of all highly symmetric line systems related to $\textrm{W}(\textrm{H}_4)$. Constant $k$ denotes the order of the linear character $\nu$ in Proposition \ref{prop:1}. Corresponding generalized highly symmetric frame has $k|\mathcal{L}|$ vectors.}
	\begin{tabular}{r|r|r}
		\toprule
		$|\mathcal{L}|$ & k & $\mathcal{A}(\mathcal{L})$ \\
		\midrule
		60     & 2    &$\{\frac{\sqrt{5}+1}{4},\frac{\sqrt{5}-1}{4},\frac{1}{2},0\}$\\ 
		144*  & 10  &$\{\sqrt{ \frac{5+\sqrt{5}}{10}}, \frac{5+\sqrt{5}}{10},\sqrt{\frac{5-\sqrt{5}}{10}},\frac{1}{\sqrt{5}}, \frac{5-\sqrt{5}}{10},0\}$\\
		300   & 2    &$15$-angular\\
		360   & 2    &$18$-angular\\
		400*  & 6    &$14$-angular\\
		480*  & 30  &$21$-angular\\
		600   & 2    &$32$-angular\\
		720*  &20   &$30$-angular\\
		900*  &4    &$32$-angular\\
		1200*&12  &$50$-angular\\		
		\bottomrule
	\end{tabular}
\end{table}

Highly symmetric line systems related to other exceptional irreducible complex reflection groups can also be obtained. For example, the smallest highly symmetric line system associated with ST(34) and unrelated to highly symmetric frames contains 17010 lines represented by complex regular 12-gons, it is $23$-angular with coherence $\mu=\frac{1}{4}\sqrt{7+4\sqrt{3}}$. Similarly, the smallest highly symmetric line system associated with $W(E_8)$ and unrelated to  highly symmetric frames is $\{\frac{2}{3},\frac{1}{\sqrt{3}},\frac{1}{3},0\}$-angular and contains 2240 lines  representad by complex regular hexagons. 

\subsection{Highly symmetric systems of subspaces from highly symmetric line systems} Take $G\cong\textrm{PSp}(4,5)$, a finite simple group of order 4680000 and let $\rho:G\rightarrow\OO(\mathbb{R}^{13})$ be one of its irreducible representations. This representation miss harmonic invariants of degrees lower than 6 meaning that any $\rho(G)$-orbit of a real vector is always a real spherical 5-design and the corresponding line system is always real projective 2-design (see \cite{DGS1,Hog,RS} for more information on spherical and projective designs). 

The smallest highly symmetric line system $\mathcal{L}$ associated with $\rho$ has 156 lines and is $\{1/5,1/\sqrt{5}\}$-angular, it is known from \cite{MW} as a member of more general infinite family of real line systems associated with finite symplectic groups. Unfortunately, this is the only member of this family that is real biangular projective 2-design. Stabilizer subgroup $H$ of a line $l\in\mathcal{L}$ is isomorphic to $5^3:(2\times A_5).2$ (Atlas notation) -- a maximal subgroup of $G$. Subgroup $H$ has a single linear character of order 2 defining a twisted spherical function with following values on each representative of $(H,H)$-double coset of $G$
\begin{displaymath}
[1,30000][\frac{1}{\sqrt{5}},900000][\frac{1}{5},3750000].
\end{displaymath}
Turns out that the Gram matrix associated with the twisted spherical function above also describes a highly symmetric system of 156 3-dimensional subspaces. To see this, take another maximal subgroup of $G$, isomorphic to $H_1\cong 5^{1+2}:4A_5$ in Atlas notation. Let $\nu$ be the character of the irreducible $3$-dimensional subrepresentation occuring in $\rho\downarrow_{H_1}^G$ with multiplicity 1. One of the highly symmetric frames associated with $\nu$  is the regular icosahedron. By "spinning" this icosahedron as in \ref{prop:1} we obtain a $\rho(G)$-frame describing 156 "spinned" 3-dimensional $H_1$-invariant subspaces. Coincidently, vectors of this frame define $\mathcal{L}$  meaning that it can be obtained in multiple ways.

Several other highly symmetric line systems associated with $\rho$ are also related to subspace packings. For example, the second smallest highly symmetric line system has 1560 lines and is related to the same highly symmetric system of 156 3-dimensional subspaces each subspace represented by a regular dodecahedron this time. Another highly symmetric line system of cardinality 9750 is related to a highly symmetric system of 325 4-dimensional subspaces, each subspace being represented by a 600-cell. The group $G$ is not unique in this regard, it is common that highly symmetric line systems define highly symmetric systems of subspaces.

\subsection{Kissing numbers in $\mathbb{R}^{10}$ and $\mathbb{R}^{11}$} Take $G\cong\textrm{PSU(4,2)}$ and let $\rho:G\rightarrow\U(\mathbb{C}^5)$ be an irreducible representation. By taking a maximal subgroup $H_1$ isomorphic to $3^3:S4$ (in Atlas notation) and its unique linear character of order two compute a twisted spherical function. Its values on some set of representatives of $(H_1,H_1)$-double cosets are
\begin{displaymath}
[1,648][\frac{1}{3},17496][\frac{i}{\sqrt{3}},7776],
\end{displaymath}
producing a highly symmetric frame $\Phi_1$ of cardinality 80 after homogenization. Inner products between distinct vectors of $\Phi_1$ belong to $\{\pm1/3,\pm i/\sqrt{3},-1\}$. 

Another twisted spherical function can be computed by taking a maximal subgroup $H_2$ isomorphic to $2.(A_4:A_4).2$ and one of its two linear characters of order 6. Its values on some set of representatives of $(H_1,H_1)$-double cosets are
\begin{displaymath}
[1,576][\frac{1}{2},18432][0,6912],
\end{displaymath} 
producing a highly symmetric frame $\Phi_2$ of cardinality 270. Inner products between distinct vectors of $\Phi_2$ belong to $\{\pm1/2,\pm e^{2\pi i/3}/2,\pm e^{4\pi i/3}/2,\pm e^{2\pi i/3},\pm e^{4\pi i/3},-1\}$.

Proposition \ref{prop:4} allows computation of matrix $M_{\Phi_1\Phi_2}$ modulo phase $c$. The first row of this matrix has the following values on some set of representatives of $(H_1,H_2)$-double cosets:
\begin{displaymath}
c[0,10368][\frac{1}{\sqrt{3}},15552].
\end{displaymath}
Thus, the inner products between distinct vectors of $\Phi_1$ and $\Phi_2$ belong to the set 
\begin{displaymath}
\{0,\pm c/\sqrt{3},\pm ce^{2\pi i/3}/\sqrt{3},\pm ce^{4\pi i/3}/\sqrt{3}\}.
\end{displaymath}

By considering  multiplication only by the real numbers we obtain a real vector space $\mathbb{R}^{10}$ out of the complex vector space $\mathbb{C}^5$. Moreover, the real part of the complex inner product of $\mathbb{C}^5$ defines an inner product in $\mathbb{R}^{10}$ allowing reinterpretation of 5-dimensional complex unit norm frames  as real 10-dimensional spherical codes. 

Real parts of inner products between distinct vectors of $\Phi_1\cup e^{2\pi i/3}\Phi_1\cup e^{4\pi i/3}\Phi_1$ belong to $\{0,\pm1/6,\pm1/3,\pm1/2,-1\}$. Similarly, real parts of inner products between distinct vectors of $\Phi_2$ always belong to $\{0,\pm1/4,\pm1/2,-1\}$. By multiplying $\Phi_2$ by a phase $ic$ observe, that inner products between distinct vectors of $ic\Phi_2$ and $\Phi_1\cup e^{2\pi i/3}\Phi_1\cup e^{4\pi i/3}\Phi_1$ belong to $\{0,\pm i/\sqrt{3},\pm ie^{4\pi i/3}/\sqrt{3},\pm ie^{4\pi i/3}/\sqrt{3}\}$. Therefore, $ic\Phi_2\cup\Phi_1\cup e^{2\pi i/3}\Phi_1\cup e^{4\pi i/3}\Phi_1$ defines a real 10-dimensional $\{0,\pm1/6,\pm1/4,\pm1/3,\pm1/2,-1\}$-angular spherical code of cardinality 510 improving the lower bound on the kissing number in $d=10$ from 500 \cite{NS} to 510. 

Interestingly, the complex 5-dimensional line system defined by the vectors of $\Phi_1\cup\Phi_2$ has $85$ lines and is $\{0,1/3,1/2,1\sqrt{3}\}$-angular. This line system appeared recently in \cite{BGM,MW}.

With pieces fitting so well, we may attempt to also improve the lower bound on the kissing number in $\mathbb{R}^{11}$ which is $582$ according to \cite{NS}. To achieve this, define an embedding $\Pi:\mathbb{R}^{10}\rightarrow\mathbb{R}^{11}$ by $\Pi(v)=(v,0)\textrm{ for all }v\in\mathbb{R}^{10}$ and let $e_{11}=(0,0,0,0,0,0,0,0,0,0,1)$. Now, the following 11-dimensional spherical code
\begin{displaymath}
\Pi(ic\Phi_2)\cup\Pi(\Phi_1)\cup \Pi(e^{2\pi i/3}\Phi_1)\cup(\frac{\sqrt{3}}{2}\Pi(e^{4\pi i/3}\Phi_1)\pm e_{11}/2)\cup\{\pm e_{11}\}
\end{displaymath}
contains $270+80+80+160+2=592$ vectors. Inner products between distinct vectors in this code are at most $1/2$, improving the lower bound on the kissing number in $\mathbb{R}^{11}$ to $592$. 

\subsection{Kissing number in $\mathbb{R}^{14}$} Take $G\cong U(3,3)$ and let $\rho:G\rightarrow\OO(\mathbb{R}^7)$ be an irreducible representation. By computing twisted spherical functions, among others, we find three interesting ones. These can be described by the following rows
\begin{displaymath}
[1,216][\frac{1}{3},5832]
\end{displaymath}
related to a maximal subgroup isomorphic to $3^{1+2}:8$ and its linear character of order 2,
\begin{displaymath}
[1,96][\frac{1}{2},16][\frac{1}{2},16][0,24][0,6]
\end{displaymath}
related to a maximal subgroup isomorphic to $4^2:S_3$ and its linear character of order 2, and
\begin{align*}
[1,16]&[0,16][-\frac{1}{2}+\frac{i}{2},64][-\frac{1}{4},256][-\frac{1}{2}+\frac{i}{4},256][-\frac{1}{4}-\frac{i}{4},256][-\frac{1}{2},256][-\frac{1}{4}+\frac{i}{2},256]\\
&[-\frac{i}{4},256][-\frac{1}{4},256][-\frac{i}{4},256][\frac{1}{4}-\frac{i}{4},256][0,256][-\frac{1}{4},256][\frac{i}{4},256][0,16][0,16][0,64]\\
&[-\frac{1}{2}+\frac{i}{4},256][-\frac{i}{4},256][\frac{1}{4}-\frac{i}{2},256][\frac{i}{4},256][-\frac{1}{4}+\frac{i}{4},256][\frac{i}{2},256][\frac{1}{4}-\frac{i}{4},256]\\
&[\frac{i}{2},64][0,64][\frac{1}{2},256][-\frac{1}{2}+\frac{i}{2},64][0,64][0,64][0,64][0,256][0,16][0,16][0,64]
\end{align*}
related to a certain subgroup of $G$ of order 16 which will be characterized later and one of its linear characters of order 4. Corresponding generalized highly symmetric frames $\Phi_1$, $\Phi_2$ and $\Phi_3$ are defined uniquely modulo phases containing 56, 126 and 1512 vectors respectively. Frames $\Phi_2$ and $\Phi_1$ form inner shells of $E_7$ and $E_7^*$ lattices respectively. Real parts of inner products between distinct vectors of $\Phi_3$ belong to $\{0,\pm1/4,\pm1/2,-1\}$. In order to improve the lower bound on the kissing number in $\mathbb{R}^{14}$ which is 1606 \cite{NS}, we attempt to add additional vectors to $\Phi_3$. The matrix $M_{\Phi_2\Phi_3}$ (modulo phase $c_{2,3}$) is described by the following row
\begin{displaymath}
c_{2,3}[0,768][-\frac{1}{2},768][0,192] [-\frac{1}{4}-\frac{i}{4},1536][-\frac{1}{4}-\frac{i}{4},1536][\frac{1}{2},768][\frac{1}{2}-\frac{i}{2},192][0,192][0,96]
\end{displaymath}
and the matrix $M_{\Phi_1\Phi_3}$ (modulo phase $c_{1,3}$) is described by the following row
\begin{displaymath}
c_{1,3}[0,864][0,864][\frac{i}{\sqrt{6}},3456][\frac{1+i}{\sqrt{6}},864].
\end{displaymath}
From these matrices we see that we may add to $\Phi_3$ either two copies of $\Phi_2$ or two copies of $\Phi_1$ multiplied by properly chosen phases (similarly as we did in the previous example) obtaining 14-dimensional kissing configurations of $1512+2\cdot126=1764$ and $1512+2\cdot56=1624$ vectors respectively. Both codes can be further improved. We will improve the first code and to do so we need to better understand the frame $\Phi_3$ and its connection to the E7 root system represented by the vectors of $\Phi_2$. The vectors of $\Phi_2$ can be grouped into vertices of 9 distinct cross polytopes, one of which, $\mathcal{B}_0$, can be fixed to be  represented by all permutations of 
\begin{displaymath}
(\pm1,0,0,0,0,0,0).
\end{displaymath}
The vertices of the remaining 8 polytopes $\mathcal{B}_1,\ldots,\mathcal{B}_8$ representing remaining vectors of $\Phi_2$ can be represented by cyclic shifts of the vectors
\begin{displaymath}
(0,0,\pm1/2,0,\pm1/2,\pm1/2,\pm1/2),
\end{displaymath}
where 7 distinct patterns on zeros form a 2-$(7,3,1)$-design. Now, modulo a phase, $\Phi_3$ is formed by 1512 unit vectors of the form
\begin{displaymath}
\frac{i^k}{\sqrt{2}}(v+iw)\textrm{ where } 0\leq k\leq3\textrm{, }v\neq w\textrm{ and }v,w\in\mathcal{B}_j\textrm{ for }0\leq j\leq8. 
\end{displaymath}
The line stabilizer of the vector $\frac{1}{\sqrt{2}}(v+iw)$ is generated by a subgroup of $G$ order 8 that stabilizes all lines on the plane spanned by the vectors $v$ and $w$ and by an additional symmetry that maps $v$ to $-w$.

Let $\Phi_4$ be the scaled copy of the $D_7$ root system represented by all permutations of
\begin{displaymath}
(\pm\frac{1}{\sqrt{2}},\pm\frac{1}{\sqrt{2}},0,0,0,0,0).
\end{displaymath}
It is easy to check that the real 14-dimensional spherical code 
\begin{displaymath}
\Phi_3\cup\frac{1+i}{\sqrt{2}}\Phi_2\cup \frac{1-i}{\sqrt{2}}\Phi_2\cup\Phi_4\cup i\Phi_4
\end{displaymath}
is $\{0,\pm1/4,\pm1/2,-1\}$-angular and contains $1512+126+126+84+84=1932$ vectors.

The following table summarizes lower and upper bounds on the kissing number in some low-dimensional spaces.
\begin{table}[htbp]%
	\centering
	\caption{Bounds on kissing number in few low-dimensional spaces.}
	\begin{tabular}{r|r|r|r}
		\toprule
		$d$ & Lower bound \cite{NS} & Improved lower bound & Upper bound \cite{MOF}\\
		\midrule
		8     & 240    &-&240\\ 
		9  & 306  &-&363\\
		10   & 500    &510&553\\
		11  & 582    &592&869\\
		12  & 840    &-&1356\\
		13  & 1154  &-&2066\\
		14   & 1606    &1932&3177\\
		15  &2564   &-&4858\\
		16  &4320    &-&7332\\	
		\bottomrule
	\end{tabular}
\end{table}

\subsection{78 lines in $d=11$} This example is related to Example 4.3 of \cite{IJM} where an $\{1/3,0\}$-angular line system $\mathcal{L}_1$ of cardinality 66  is constructed in $\mathbb{R}^{11}$ from reducible representation of the Mathieu group $M_{11}$. The same line system can be constructed as a highly symmetric line system from 11-dimensional (permutation) representation of the Mathieu group $M_{12}$ by using as line stabilizer a subgroup isomorphic to $H_1\cong M_{10}.2$ (in Atlas notation) permuting transitively 12 points. Another highly symmetric line system $\mathcal{L}_2$ associated with 11 dimensional representation of $M_{12}$ is spanned by 12 vertices of the regular 11-simplex. Within $M_{12}$ each line of the 11-simplex is stabilized by a subgroup isomorphic to $H_2\cong M_{11}$ fixing 1 point in the permutation representation of $M_{12}$.

There is only one $(H_1,H_2)$-double coset inside $M_{12}$: the whole $M_{12}$. Thus, there is only one angle between lines of $\mathcal{L}_1$ and lines of $\mathcal{L}_2$. Since highly symmetric line systems are represented by tight frames, this angle has to be $1/\sqrt{11}$. Thus, $\mathcal{L}_1\cup\mathcal{L}_2$ is an $\{1/3,1/\sqrt{11},1/11,0\}$-angular line system of cardinality 78. According to Levenshtein's second bound \cite{L}, the coherence of any real 11-dimensional line system of cardinality $78$ is at least $\sqrt{7/67}\approx0.323$ meaning that coherence of $\mathcal{L}_1\cup\mathcal{L}_2$ is very close to the lower bound.

\section{Acknowledgements}
I would like to thank Professor Patric Osterg\aa rd for valuable comments.

\end{document}